\documentclass[11pt]{amsart}
\usepackage{amsmath,amsthm, amssymb, amsfonts, amscd}
\usepackage[mathscr]{eucal}
\usepackage{epsfig}
\usepackage{graphicx}
\usepackage{psfrag}
\usepackage[usenames,dvipsnames]{color}
\usepackage{xy}
\usepackage{subfigure}

\newtheorem{mytheorem}{Theorem}
\newtheorem{theorem}{Theorem}[section]
\newtheorem{mycorollary}[mytheorem]{Corollary}
\newtheorem{corollary}[theorem]{Corollary}
\newtheorem{lemma}[theorem]{Lemma}
\newtheorem{proposition}[theorem]{Proposition}

\newtheorem*{PhiBettinumbers}{\thmref{thm:PhiBettinumbers}}
\newtheorem*{Psihomology}{\thmref{thm:Psihomology}}
\newtheorem*{echo}{Corollary~\ref{cor:echo}}
\newtheorem*{surfacecohomology}{Corollary~\ref{cor:surfacecohomology}}
\newtheorem*{n-1cohomologyPhi}{Corollary~\ref{cor:n-1cohomologyPhi}}

\theoremstyle{definition}

\theoremstyle{definition}
\newtheorem{remark}[theorem]{Remark}

\newcommand{\thmref}[1]{Theorem~\ref{#1}}
\newcommand{\propref}[1]{Proposition~\ref{#1}}

\newcommand{\lemref}[1]{Lemma~\ref{#1}}

\newcommand{\hk}{{\mathcal H}^k}

\input epsf

\newcommand{\til}[1]{{\widetilde{#1}}}
\makeatletter
\def\imod#1{\allowbreak\mkern10mu({\operator@font mod}\,\,#1)}
\makeatother

\input xy
\xyoption{all}

\usepackage[colorlinks=true, urlcolor=MidnightBlue, linkcolor=MidnightBlue, citecolor=MidnightBlue, pdftitle={The complete Dirichlet-to-Neumann map for differential forms}, pdfauthor={Vladimir Sharafutdinov, Clayton Shonkwiler}, pdfsubject={Riemannian geometry, inverse problems}, pdfkeywords={Hodge theorem, Dirichlet-to-Neumann map, Hilbert transform}]{hyperref}

\begin{document}

\title[The complete Dirichlet-to-Neumann map]{The complete Dirichlet-to-Neumann map for differential forms}  

\author{Vladimir Sharafutdinov}
\address{Sobolev Institute of Mathematics}
\email{sharaf@math.nsc.ru}
\urladdr{\href{http://www.math.nsc.ru/~sharafutdinov/}{http://www.math.nsc.ru/~sharafutdinov/}}

\author{Clayton Shonkwiler}
\address{Department of Mathematics \\ Haverford College}
\email{cshonkwi@haverford.edu}
\urladdr{\href{http://www.haverford.edu/math/cshonkwi/}{http://www.haverford.edu/math/cshonkwi/}}

\date{\today}
\maketitle

\begin{abstract}
	The Dirichlet-to-Neumann map for differential forms on a Riemannian manifold with boundary is a generalization of the classical Dirichlet-to-Neumann map which arises in the problem of Electrical Impedance Tomography. We synthesize the two different approaches to defining this operator by giving an invariant definition of the complete Dirichlet-to-Neumann map for differential forms in terms of two linear operators $\Phi$ and $\Psi$. The pair $(\Phi, \Psi)$ is equivalent to Joshi and Lionheart's operator $\Pi$ and determines Belishev and Sharafutdinov's operator $\Lambda$. We show that the Betti numbers of the manifold are determined by $\Phi$ and that $\Psi$ determines a chain complex whose homologies are explicitly related to the cohomology groups of the manifold.
\end{abstract}

\section{Introduction} 
\label{sec:introduction}

We consider the problem of recovering the topology of a compact, oriented, smooth Riemannian manifold $(M,g)$ with boundary from the Dirichlet-to-Neumann map for differential forms. The classical Dirichlet-to-Neumann map for functions was first defined by Calder\'on \cite{Calderon}, and has been shown to recover surfaces up to conformal equivalence \cite{LassasUhlmann, Belishev1} and real-analytic manifolds of dimension $\geq 3$ up to isometry \cite{LTU}.

The classical Dirichlet-to-Neumann map was generalized to an operator on differential forms independently by Joshi and Lionheart \cite{Joshi} and Belishev and Sharafutdinov \cite{BS}. Joshi and Lionheart called their operator $\Pi$ and showed that the data $(\partial M, \Pi)$ determines the $C^\infty$-jet of the Riemannian metric at the boundary. Krupchyk, Lassas, and Uhlmann have recently extended this result to show that $(\partial M, \Pi)$ determines a real-analytic manifold up to isometry \cite{KLU}.

On the other hand, Belishev and Sharafutdinov called their Dirichlet-to-Neumann map $\Lambda$ and showed that $(\partial M, \Lambda)$ determines the cohomology groups of the manifold $M$. Shonkwiler \cite{shonk} demonstrated a connection between $\Lambda$ and invariants called Poincar\'e duality angles and showed that the cup product structure of the manifold $M$ can be partially recovered from $(\partial M, \Lambda)$.

The operators $\Pi$ and $\Lambda$ are similar, but do not appear to be equivalent. One of the advantages of Belishev and Sharafutdinov's $\Lambda$, especially for the task of recovering topological data, is that it is defined invariantly. In this paper we provide an invariant definition of Joshi and Lionheart's operator $\Pi$, which we give in terms of two auxiliary operators
\[
	\Phi: \Omega^k(\partial M) \to \Omega^{n-k-1}(\partial M) \quad \text{and} \quad \Psi: \Omega^k(\partial M) \to \Omega^{k-1}(\partial M).
\]
We can easily show that $\Lambda$ is determined by $\Phi$ and $\Psi$, so it makes sense to regard $\Pi$ as the ``complete'' Dirichlet-to-Neumann operator on differential forms.

Belishev and Sharafutdinov's proof that the Betti numbers of $M$ can be recovered from the data $(\partial M, \Lambda)$ was somewhat circuitous, as it involved determining the dimension of the image of the operator $G= \Lambda \pm d_\partial \Lambda^{-1} d_\partial$. In contrast, it is straightforward to recover the Betti numbers of $M$ from $\Phi$.

\begin{mytheorem}\label{thm:PhiBettinumbers}
	Let $\beta_k(M) = \dim H^k(M; \mathbb{R})$ be the $k$th Betti number of $M$. Then
	\[
		\beta_k(M) = \dim \ker \Phi.
	\]
\end{mytheorem}

The operator $\Psi$ turns out to be a chain map and the homology of the chain complex $(\Omega^*(\partial M), \Psi)$ is given in terms of a mixture of absolute and relative cohomology groups of $M$.

\begin{mytheorem}\label{thm:Psihomology}
	For any $0 \leq k \leq n-1$,
	\[
		H_k(\Omega^*(\partial M), \Psi)  \simeq H^{k+1}(M, \partial M; \mathbb{R}) \oplus H^k(M; \mathbb{R}).
	\]
\end{mytheorem}

This, in turn, implies that the space of $k$-forms on $\partial M$ contains an ``echo'' (detected by $\Pi$) of the $(k+1)$st relative cohomology group of $M$.

\begin{mycorollary}\label{cor:echo}
	The space $\Omega^k(\partial M)$ of $k$-forms on $\partial M$ contains a subspace isomorphic to $H^{k+1}(M, \partial M; \mathbb{R})$ which is distinguished by the Dirichlet-to-Neu\-mann operator $\Pi$. Specifically,
	\[
		(\ker \Psi_k/\mathrm{im}\, \Psi_{k+1})/\ker \Phi_k \simeq H^{k+1}(M, \partial M; \mathbb{R}).
	\]
\end{mycorollary}

When $n = 2$ and $k = 0$, \thmref{thm:PhiBettinumbers} and Corollary~\ref{cor:echo} imply that all the cohomology groups of a surface are contained in $\Omega^0(\partial M)$.
	
\begin{mycorollary}\label{cor:surfacecohomology}
	All of the cohomology groups of a surface $M$ with boundary can be realized inside the space of smooth functions on $\partial M$, where they can be recovered by the Dirichlet-to-Neumann operator $\Pi$.
\end{mycorollary}

Since $\Psi$ is a chain map, it is natural to try to define associated cochain maps and compute their cohomologies. In this spirit, we define $\til{\Psi} = \pm \star_\partial \Psi \star_\partial$ and show that it is the adjoint of $\Psi$. Not surprisingly,
\[
	H^k(\Omega^*(\partial M), \til{\Psi}) \simeq H_{n-k-1}(\Omega^*(\partial M), \Psi).
\]

Finally, we define another cochain map $\Theta$ with the same cohomology as $\til{\Psi}$. It turns out that $\Theta = \pm d_\partial \Phi^2$, so the cohomology of $\til{\Psi}$ (and hence the homology of $\Psi$) is completely determined by the operator $\Phi$. With this in mind, restating Corollary~\ref{cor:echo} in terms of $\Phi$ and specializing to the case $k=0$ yields the following:

\begin{mycorollary}\label{cor:n-1cohomologyPhi}
	A copy of the cohomology group $H^{n-1}(M; \mathbb{R})$ is distinguished by the operator $\Phi$ inside $\Omega^0(\partial M)$, the space of smooth functions on $\partial M$. Specifically,
	\[
		\ker(d_\partial \Phi^2)/\ker \Phi \simeq H^{n-1}(M; \mathbb{R}).
	\]
\end{mycorollary}

The above results all suggest that the operator $\Pi$ (and, in particular, $\Phi$) encodes more information about the topology of $M$ than does the operator $\Lambda$. Thus far nobody has been able to use $\Lambda$ to recover the cohomology ring structure on $M$, but perhaps this will be easier to recover from the operator $\Pi$. Another interesting question relates to the linearized inverse problem of recovering the metric: can the results of \cite{Sh} be strengthened if the data $\Lambda$ are replaced with the richer data $(\Phi,\Psi)$?


\section{The operators $\Phi$ and $\Psi$} 
\label{sec:preliminaries}
Throughout this paper, $(M,g)$ will be a smooth, compact, oriented Riemannian manifold of dimension $n \geq 2$ with nonempty boundary. The term ``smooth'' is used as a synonym for ``$C^\infty$-smooth''. Let $i: \partial M \hookrightarrow M$ be the identical embedding and let $\Omega(M) = \bigoplus_{k=0}^n \Omega^k(M)$ be the graded algebra of smooth differential forms on $M$. We use the standard operators $d, \delta, \Delta$, and $\star$ on $\Omega(M)$, as well as their analogues $d_\partial, \delta_\partial, \Delta_\partial$, and $\star_\partial$ on $\Omega(\partial M)$.

Joshi and Lionheart defined their Dirichlet-to-Neumann operator
\[
	\Pi: \Omega(M)|_{\partial M} \to \Omega(M)|_{\partial M}
\]
as
\[
	\Pi \chi := \left. \frac{\partial \omega}{\partial \nu} \right|_{\partial M},
\]
where $\nu$ is the unit outward normal vector at the boundary and $\omega$ is the solution to the boundary value problem
\[
	\begin{cases} \Delta \omega = 0 \\ \omega|_{\partial M} = \chi. \end{cases}
\]
This boundary value problem has a unique solution for every $\chi \in \Omega(M)|_{\partial M}$ \cite[Theorem~3.4.1]{Schwarz}.

When applied to forms, the meaning of the normal derivative $\partial/\partial \nu$ needs to be specified. Instead, we prefer to give an equivalent definition of $\Pi$ in invariant terms. To do so, note that the restriction $\omega|_{\partial M}$ is determined by two boundary forms, $i^*\omega$ and $i^*\!\star \omega$. Likewise, the data $\partial \omega/\partial \nu|_{\partial M}$ are equivalent to the two boundary forms $i^*\!\star d\omega$ and $i^*\delta \omega$. Hence, we will define the operator
\[
	\Pi: \Omega^k(\partial M) \times \Omega^{n-k}(\partial M) \to \Omega^{n-k-1}(\partial M) \times \Omega^{k-1}(\partial M)
\]
by
\begin{equation}\label{eqn:Pidef}
	\Pi \begin{pmatrix} \varphi \\ \psi \end{pmatrix} = \begin{pmatrix} i^*\!\star d\omega \\ i^*\delta \omega \end{pmatrix}
\end{equation}
where $\omega \in \Omega^k(M)$ is the solution to the boundary value problem
\begin{equation}\label{eqn:PiBVP}
	\begin{cases} \Delta \omega = 0 & \\ i^*\omega = \varphi , & i^*\!\star \omega = \psi.\end{cases}
\end{equation}

Since $\Pi$ sends pairs of forms to pairs of forms, it is somewhat cumbersome to work with in practice. Instead of using it directly, we find a pair of operators $(\Phi, \Psi)$ which is equivalent to $\Pi$. Define the linear operators
\[
	\Phi: \Omega^k(\partial M) \to \Omega^{n-k-1}(\partial M) \quad \text{and} \quad \Psi: \Omega^k(\partial M) \to \Omega^{k-1}(\partial M)
\]
by the equalities
\begin{equation}\label{eqn:PhiPsidef}
	\Phi \varphi = i^*\!\star d \omega \quad \text{and} \quad \Psi \varphi = i^*\delta \omega.
\end{equation}
Here $\omega \in \Omega^k(M)$ is the solution to the boundary value problem
\begin{equation}\label{eqn:PhiPsiBVP}
	\begin{cases} \Delta \omega = 0 & \\ i^*\omega = \varphi, & i^*\!\star \omega = 0. \end{cases}
\end{equation}

Now it is straightforward to express $\Pi$ in terms of $\Phi$ and $\Psi$. We write $\Pi$ as the matrix
\[
	\Pi = \begin{pmatrix} \Pi_{11} & \Pi_{12} \\ \Pi_{21} & \Pi_{22} \end{pmatrix}.
\]

Then, comparing \eqref{eqn:Pidef} and \eqref{eqn:PhiPsidef},
\[
	\Pi_{11} = \Phi, \quad \Pi_{21} = \Psi.
\]
From \eqref{eqn:Pidef} and \eqref{eqn:PiBVP}, the operators $\Pi_{12}$ and $\Pi_{22}$ are given by
\[
	\Pi_{12} \psi = i^*\!\star d \varepsilon \quad \text{and} \quad \Pi_{22}\psi = i^*\delta \varepsilon,
\]
where $\varepsilon$ solves the boundary value problem
\[
	\begin{cases} \Delta \varepsilon = 0 & \\ i^* \varepsilon = 0, & i^*\!\star \varepsilon = \psi. \end{cases}
\]

If $\varepsilon \in \Omega^{k}(M)$ is the solution to this boundary value problem for $\psi \in \Omega^{n-k}(\partial M)$, then the form $\omega = \star \varepsilon$ solves the problem
\[
	\begin{cases} \Delta \omega = 0 & \\ i^*\omega = \psi, & i^*\!\star \omega = 0. \end{cases}
\]
Comparing this to \eqref{eqn:PhiPsiBVP}, we see that
\begin{equation}\label{eqn:PhiPsiomega}
	\Phi \psi = i^*\!\star d\omega \quad \text{and} \quad \Psi \psi = i^*\delta \omega.
\end{equation}
Since
\[
	i^*\!\star d\omega = (-1)^{n(n-k)+1}i^*\delta \varepsilon \quad \text{and} \quad i^*\delta \omega = (-1)^{k+1} i^*\!\star d \varepsilon,
\]
\eqref{eqn:Pidef} and \eqref{eqn:PhiPsiomega} imply that
\[
	\Pi_{12} = (-1)^{n(n-k) + 1} \Psi \quad \text{and} \quad \Pi_{22} = (-1)^{k+1} \Phi \quad \text{on} \quad \Omega^{n-k}(\partial M).
\]

Therefore, the operator $\Pi$ can be expressed in terms of $\Phi$ and $\Psi$ as
\begin{equation}\label{eqn:PiPhiPsi}
	\Pi = \begin{pmatrix} \Phi & (-1)^{n(n-k)+1} \Psi \\ \Psi & (-1)^{k+1}\Phi \end{pmatrix} \quad \text{on} \quad \Omega^{k}(\partial M) \times \Omega^{n-k}(\partial M).
\end{equation}

Belishev and Sharafutdinov's version of the Dirichlet-to-Neumann map is the operator
\[
	\Lambda: \Omega^k(\partial M) \to \Omega^{n-k-1}(\partial M)
\]
given by
\[
	\Lambda \varphi = i^*\!\star d\omega,
\]
where $\omega \in \Omega^k(M)$ is a solution to the boundary value problem
\begin{equation}\label{eqn:LambdaBVP}
	\begin{cases} \Delta \omega = 0 & \\ i^*\omega = \varphi , & i^*\delta \omega = 0. \end{cases}
\end{equation}

We can now express the operator $\Lambda$ in terms of $\Phi$ and $\Psi$. Given $\varphi \in \Omega^k(\partial M)$, let $\omega \in \Omega^k(M)$ solve the boundary value problem \eqref{eqn:LambdaBVP} and set $\psi = i^*\!\star \omega$. Then $\omega$ solves the boundary value problem \eqref{eqn:PiBVP}, so we have that
\[
	\Pi \begin{pmatrix} \varphi \\ \psi \end{pmatrix} = \begin{pmatrix} i^*\!\star d\omega \\ i^*\delta \omega \end{pmatrix} = \begin{pmatrix} \Lambda \varphi \\ 0 \end{pmatrix}.
\]
With the help of \eqref{eqn:PiPhiPsi} we can rewrite this equation as the system
\begin{align*}
	\Phi \varphi + (-1)^{n(n-k)+1}\Psi \psi & = \Lambda \varphi \\
	\Psi \varphi + (-1)^{k+1}\Phi \psi & = 0.
\end{align*}
Eliminating $\psi$ from the system yields the expression
\begin{equation}
	\label{eqn:LambdaPhiPsi} \Lambda = \Phi + (-1)^{n(n-k)+k+1}\Psi \Phi^{-1}\Psi \quad \text{on} \quad \Omega^k(\partial M).
\end{equation}
The fact that the operator $\Psi \Phi^{-1}\Psi$ is well-defined follows from Corollary~\ref{cor:domainrangePhiPsi}, stated below.

We take this opportunity to record some useful relations involving $\Phi$ and $\Psi$:

\begin{lemma}\label{lem:PhiPsiproperties}
	The operators $\Phi$ and $\Psi$ satisfy the following relations:
	\begin{align}
		\Phi \Psi & = (-1)^k d_\partial \Phi \quad \text{on} \quad \Omega^k(\partial M), \label{eqn:PhiPsi} \\
		\Psi^2 & = 0 \label{eqn:Psi2} \\
		\Psi \Phi & = (-1)^{k+1} \Phi d_\partial \quad \text{on} \quad \Omega^k(\partial M), \label{eqn:PsiPhi} \\
		\Phi^2 & = (-1)^{kn}(d_\partial \Psi + \Psi d_\partial) \quad \text{on} \quad \Omega^k(\partial M) \label{eqn:Phi2}
	\end{align}
\end{lemma}

\begin{proof}
	Given $\varphi \in \Omega^k(\partial M)$, let $\omega \in \Omega^k(M)$ solve the boundary value problem \eqref{eqn:PhiPsiBVP}. Then
	\begin{equation}\label{eqn:PhiPsidef1}
		\Phi \varphi = i^*\!\star \omega, \quad \Psi \varphi = i^*\delta \omega.
	\end{equation}
	Letting $\xi = \delta \omega$, we certainly have $\Delta \xi = 0$. Pulling $\xi$ and $\star \xi$ back to the boundary yields
	\begin{align*}
		i^*\xi & = i^*\delta \omega = \Psi \varphi \\
		i^*\!\star \xi & = i^*\!\star \delta \omega = \pm i^*d \star \omega = \pm d_\partial i^*\!\star \omega = 0.
	\end{align*}
	
	Therefore, $\xi$ solves the boundary value problem
	\[
		\begin{cases} \Delta \xi = 0 & \\ i^*\xi = \Psi \varphi, & i^*\!\star \xi = 0, \end{cases}
	\]
	and so
	\begin{equation}\label{eqn:PhiPsiPsi2}
		\Phi \Psi \varphi = i^*\!\star d\xi \quad \text{and} \quad \Psi^2 \varphi = i^*\delta \xi.
	\end{equation}
	Since $\Delta \omega = 0$, it follows that $d\delta \omega = -\delta d \omega$, which we use to see that
	\begin{align*}
		i^*\!\star d\xi & = i^*\!\star d\delta \omega = -i^*\!\star \delta d\omega = (-1)^k i^*d \star d\omega = (-1)^k d_\partial i^* \!\star d\omega, \\
		i^*\delta \xi & = i^*\delta \delta \omega = 0.
	\end{align*}
	
	Comparing this with \eqref{eqn:PhiPsiPsi2}, we obtain
	\[
		\Phi\Psi \varphi = (-1)^k d_\partial i^*\!\star d\omega \quad \text{and} \quad \Psi^2 \varphi = 0.
	\]
	With the help of \eqref{eqn:PhiPsidef1}, this gives \eqref{eqn:PhiPsi} and \eqref{eqn:Psi2}.
	
	Turning to \eqref{eqn:PsiPhi}, we again let $\omega \in \Omega^k(M)$ solve \eqref{eqn:PhiPsiBVP} for a form $\varphi \in \Omega^k(\partial M)$. Let $\varepsilon \in \Omega^{k+1}(M)$ be a solution to the problem
	\[
		\begin{cases} \Delta \varepsilon = 0 & \\ i^*\varepsilon = d \varphi, & i^*\!\star \varepsilon = 0. \end{cases}
	\]
	Then
	\begin{equation}\label{eqn:PhidPsid}
		\Phi d_\partial \varphi = i^*\!\star d \varepsilon, \quad \Psi d_\partial \varphi = i^*\delta \varepsilon.
	\end{equation}
	
	Define $\eta \in \Omega^{n-k-1}(M)$ by
	\begin{equation}\label{eqn:etadef}
		\eta = \star d \omega - \star \varepsilon.
	\end{equation}
	Clearly, $\Delta \eta = 0$. Moreover,
	\[
		\star \eta = \star \star (d\omega - \varepsilon) = \pm (d\omega - \varepsilon),
	\]
	so
	\[
		i^*\!\star \eta = \pm i^*(d\omega - \varepsilon) = \pm (d \varphi - d \varphi) = 0.
	\]
	Also,
	\[
		i^* \eta = i^*\!\star d\omega - i^*\!\star \varepsilon = \Phi \varphi,
	\]
	since $i^*\!\star \varepsilon = 0$.
	
	Therefore, $\eta$ solves the boundary value problem
	\[
		\begin{cases} \Delta \eta = 0 & \\ i^* \eta = \Psi \varphi, & i^*\!\star \eta = 0. \end{cases}
	\]
	Hence,
	\begin{equation}\label{eqn:Phi2PsiPhi}
		\Phi^2 \varphi= i^*\!\star d\eta \quad \text{and} \quad \Phi \Psi \varphi = i^*\delta \eta.
	\end{equation}
	Using \eqref{eqn:etadef} we see that
	\[
		\delta \eta = \delta \star d\omega - \delta \star \varepsilon = \pm \star dd\omega - \delta \star \varepsilon = (-1)^{k+1} \star d \varepsilon.
	\]
	Thus,
	\[
		i^*\delta \eta = (-1)^{k+1} i^*\!\star d \epsilon,
	\]
	which, along with \eqref{eqn:PhidPsid} and \eqref{eqn:Phi2PsiPhi}, yields
	\[
		\Psi \Phi \varphi = (-1)^{k+1} \Phi d_\partial \varphi,
	\]
	proving \eqref{eqn:PsiPhi}.
	
	Finally, \eqref{eqn:Phi2} is proved along the same lines. From \eqref{eqn:etadef} we have
	\[
		\star d \eta = \star\, d \star (d\omega - \varepsilon) = (-1)^{kn+1} (\delta d\omega - \delta \varepsilon).
	\]
	Again making use of the fact that $\delta d \omega = -d \delta \omega$, this implies that
	\[
		i^*\!\star d\omega = (-1)^{kn+1}\left(i^*\delta d\omega - i^*\delta \varepsilon\right) = (-1)^{kn} \left( d_\partial i^*\delta \omega + i^*\delta \varepsilon \right).
	\]
	In turn, we can use \eqref{eqn:PhiPsidef1} and \eqref{eqn:PhidPsid} to rewrite the above formula as
	\[
		i^*\!\star d\eta = (-1)^{kn} \left(d_\partial \Psi \varphi + \Psi d_\partial \varphi \right).
	\]
	Comparing with \eqref{eqn:Phi2PsiPhi}, this produces the desired relation \eqref{eqn:Phi2}.
\end{proof}

\begin{remark}\label{rem:Lambdaproperties}
	The key properties of the operator $\Lambda$ are expressed by the equalities
	\[
		\Lambda d_\partial = 0, \quad d_\partial \Lambda = 0, \quad \text{and} \quad \Lambda^2 = 0.
	\]
	It is straightforward to check that these equalities follow from \eqref{eqn:LambdaPhiPsi} and \lemref{lem:PhiPsiproperties}.
\end{remark}


\section{Recovering the Betti numbers of $M$ from $\Phi$} 
\label{sec:betti_numbers}
Belishev and Sharafutdinov showed that the Betti numbers of the manifold $M$,
\[
	\beta_k(M) = \dim H^k(M; \mathbb{R}),
\]
can be recovered from the data $(\partial M, \Lambda)$. The proof of this fact is somewhat indirect, involving the auxiliary operator
\begin{equation}\label{eqn:Gdef}
	G = \Lambda + (-1)^{kn+k+n} d_\partial \Lambda^{-1} d_\partial : \Omega^k(\partial M) \to \Omega^{n-k-1}(\partial M).
\end{equation}

In contrast, it is much more straightforward to recover the Betti numbers of $M$ from the operator $\Phi$.

\begin{PhiBettinumbers}
	Let $\Phi_k: \Omega^k(\partial M) \to \Omega^{n-k-1}(\partial M)$ be the restriction of $\Phi$ to $\Omega^k(\partial M)$. Then
	\[
		\beta_k(M) = \dim \ker \Phi_k.
	\]
\end{PhiBettinumbers}

The Hodge--Morrey--Friedrichs decomposition theorem \cite[Section~2.4]{Schwarz} implies that
\[
	H^k(M; \mathbb{R}) \simeq \hk_N(M),
\]
where
\[
	\hk_N(M) := \{\omega \in \Omega^k(M) : d\omega = 0, \delta \omega = 0, i^*\!\star \omega = 0\}
\]
is the space of harmonic Neumann fields. Since harmonic forms are uniquely determined by their boundary values, $\hk_N(M) \simeq i^*\hk_N(M)$, so \thmref{thm:PhiBettinumbers} is an immediate consequence of the following lemma.

\begin{lemma}\label{lem:kerPhi}
	The kernel of the operator $\Phi_k: \Omega^k(\partial M) \to \Omega^{n-k-1}(\partial M)$ consists of the boundary traces of harmonic Neumann fields; i.e.,
	\[
		\ker \Phi_k = i^*\hk_N(M).
	\]
	The image of $\Phi_k$ coincides with the subspace $(i^*\hk_N(M))^\perp \subset \Omega^{n-k-1}(\partial M)$ consisting of forms $\psi \in \Omega^{n-k-1}(\partial M)$ satisfying
	\begin{equation}\label{eqn:imPhi}
		\int_{\partial M} \psi \wedge \chi = 0 \quad \forall \xi \in i^*\hk_N(M).
	\end{equation}
	In particular, $\Phi$ is a Fredholm operator with index zero.
\end{lemma}

\begin{proof}
	If $\varphi \in \Omega^k(\partial M)$ such that $\Phi_k \varphi = 0$, then the boundary value problem
	\begin{equation}\label{eqn:kerPhiBVP}
		\begin{cases} \Delta \omega = 0 & \\ i^*\omega = \varphi, & i^*\!\star \omega = 0, \quad i^*\!\star d\omega = 0 \end{cases}
	\end{equation}
	is solvable. Using Green's formula,
	\[
		\langle d\omega, d\omega\rangle_{L^2} + \langle \delta \omega, \delta \omega \rangle_{L^2} = \langle \Delta \omega, \omega \rangle_{L^2} + \int_{\partial M} i^*(\omega \wedge \star d\omega - \delta \omega \wedge \star \omega).
	\]
	The right side of this equation equals zero since $\omega$ solves the boundary value problem \eqref{eqn:kerPhiBVP}. Hence, $\omega$ is a harmonic Neumann field since $i^*\!\star \omega = 0$, and so $\varphi = i^*\omega \in i^*\hk_N(M)$.
	
	The converse statement is immediate: if $\varphi = i^*\omega$ for $\omega \in \hk_N(M)$, then $\omega$ solves the boundary value problem \eqref{eqn:kerPhiBVP} and hence $\varphi \in \ker \Phi_k$.
	
	On the other hand, a form $\psi \in \Omega^{n-k-1}(\partial M)$ is in the image of $\Phi_k$ if and only if the boundary value problem
	\[
		\begin{cases} \Delta \omega = 0 & \\ i^*\!\star \omega = 0, & i^*\!\star d\omega = \psi \end{cases}
	\]
	is solvable. The defining condition \eqref{eqn:imPhi} of $(i^*\hk_N(M))^\perp$ is precisely the necessary and sufficient condition for the solvability of this boundary value problem \cite[Corollary~3.4.8]{Schwarz}.
\end{proof}

\begin{corollary}\label{cor:dPhi-1}
	The operator $d_\partial \Phi^{-1}$ is well-defined on $\mathrm{im}\, \Phi_k = (i^*\hk_N(M))^\perp$; i.e., the equation $\Phi \varphi = \psi$ has a solution $\varphi$ for every $\psi \in (i^*\hk_N(M))^\perp$ and $d_\partial \varphi$ is uniquely determined by $\psi$.
\end{corollary}

\begin{proof}
	A form $\psi \in (i^*\hk_N(M))^\perp$ belongs to the range of $\Phi$, so the equation $\Phi \varphi = \psi$ is solvable. If $\Phi \varphi_1 = \Phi \varphi_2$, then the form $\varphi_1 - \varphi_2 \in \ker \Phi$ is closed, meaning that $d_\partial \varphi_1 = d_\partial \varphi_2$.
\end{proof}

The apparent similarity between the operator $d_\partial \Phi^{-1}$ and the Hilbert transform $T = d_\partial \Lambda^{-1}$ defined by Belishev and Sharafutdinov is no accident, as the following proposition demonstrates. Thus, the connection to the Poincar\'e duality angles of $M$ \cite[Theorem 4]{shonk} comes directly from the definition of $\Phi$ (and hence $\Pi$) without using $\Lambda$ as an intermediary.

\begin{proposition}\label{prop:dPhi-1}
	$d_\partial \Lambda^{-1} = d_\partial \Phi^{-1}$, where the term on the right-hand side is understood to be the restriction of $d_\partial \Phi^{-1}$ to $\mathrm{im}\, \Lambda = i^*\hk(M)$.
\end{proposition}

\begin{proof}
	Suppose $\varphi \in \text{im}\, \Lambda = i^*\hk(M)$. Then $\varphi = i^*\omega$ for some $\omega \in \hk(M)$. The Friedrichs decomposition says that
	\[
		\hk(M) = c\mathcal{EH}^k(M) \oplus \hk_D(M),
	\]
	where
	\begin{align*}
		c\mathcal{EH}^k(M) & = \{\delta \xi \in \Omega^k(M): d\delta \xi = 0\} \\
		\hk_D(M) & = \{\eta \in \Omega^k(M) : d\eta = 0, \delta \eta = 0, i^*\eta = 0\}.
	\end{align*}
	Hence,
	\[
		\omega = \delta \xi + \eta \in c\mathcal{EH}^k(M) \oplus \hk_D(M).
	\]
	The form $\xi \in \Omega^{k+1}(M)$ can be chosen such that $\xi$ is closed, $\Delta \xi = 0$, and $i^*\xi =0$ \cite[p. 87, Remark 2]{Schwarz}. Therefore,
	\[
		\begin{cases} \Delta \star \xi = 0, \\ i^*\!\star(\star\, \xi) = 0, \\ i^*\delta \star \xi = \pm i^*\!\star d \star \star\, \xi = \pm i^*\!\star d\xi = 0. \end{cases}
	\]
	This implies that $\star\, \xi$ solves the boundary value problems associated to both $\Lambda$ and $\Phi$, so
	\[
		\Lambda i^*\!\star \xi = i^*\!\star d \star \xi = (-1)^{nk+1} i^*\delta \xi = (-1)^{nk+1}i^*\omega = (-1)^{nk+1}\varphi
	\]
	and
	\[
		\Phi i^*\!\star \xi = i^*\!\star d \star \xi = (-1)^{nk+1} i^*\delta \xi = (-1)^{nk+1}i^*\omega = (-1)^{nk+1} \varphi.
	\]
	Hence,
	\[
		d\Lambda^{-1} \varphi = (-1)^{nk+1} d\, i^*\!\star \xi = d\Phi^{-1}i^*\!\star \xi,
	\]
	so we conclude that, indeed, $d\Lambda^{-1} = d\Phi^{-1}$.
\end{proof}


\section{The homology of the chain complex $(\Omega^*(\partial M), \Psi)$} 
\label{sec:homology_psi}

We saw in \lemref{lem:PhiPsiproperties} that $\Psi^2 = 0$, so it is natural to ask: what is the homology of the chain complex $(\Omega^*(\partial M), \Psi)$? 

\begin{Psihomology}
	For any $0 \leq k \leq n-1$, if $\Psi_k: \Omega^k(\partial M) \to \Omega^{k-1}(\partial M)$ is the restriction of $\Psi$ to the space of $k$-forms on $\partial M$, then
	\[
		H_k(\Omega^*(\partial M), \Psi) = \frac{\ker \Psi_k}{\mathrm{im}\, \Psi_{k+1}} \simeq H^{k+1}(M, \partial M; \mathbb{R}) \oplus H^k(M; \mathbb{R}).
	\]
\end{Psihomology}

In other words, the homology groups of $(\Omega^*(\partial M), \Psi)$ contain the absolute cohomology groups of $M$ in the same dimension and echoes of the relative cohomology groups of $M$ in one higher dimension. This behavior is similar to that exhibited by the cohomology of harmonic forms studied by Cappell, DeTurck, Gluck, and Miller \cite{CDGM}.

Since $H^k(M; \mathbb{R}) \simeq \ker \Phi_k$ (by \thmref{thm:PhiBettinumbers}) and since it will turn out that $\mathrm{im}\, \Psi_{k+1}$ completely misses $\ker \Phi_k$, we can see the echo of the $(k+1)$st relative cohomology group of $M$ inside the space of $k$-forms on $\partial M$.

\begin{echo}
	The space $\Omega^k(\partial M)$ of $k$-forms on $\partial M$ contains a space isomorphic to $H^{k+1}(M, \partial M; \mathbb{R})$ which is distinguished by the Dirichlet-to-Neu\-mann operator $\Pi$. Specifically,
	\[
		(\ker \Psi_k/\mathrm{im}\, \Psi_{k+1})/\ker \Phi_k \simeq H^{k+1}(M, \partial M; \mathbb{R}).
	\]
\end{echo}

When $n = 2$ and $k = 0$, \thmref{thm:PhiBettinumbers} and Corollary~\ref{cor:echo} imply that $H^0(M; \mathbb{R})$ and $H^1(M, \partial M; \mathbb{R})$ can be distinguished inside the space of functions on $\partial M$. Moreover, by Poincar\'e--Lefschetz duality, $H^0(M; \mathbb{R}) \simeq H^2(M, \partial M; \mathbb{R})$ and $H^1(M, \partial M; \mathbb{R}) \simeq H^1(M; \mathbb{R})$. Since $H^0(M, \partial M; \mathbb{R})$ and $H^2(M; \mathbb{R})$ are both trivial, we have the following corollary.
	
\begin{surfacecohomology}
	All of the cohomology groups of a surface $M$ with boundary can be realized inside the space of smooth functions on $\partial M$, where they can be recovered by the Dirichlet-to-Neumann operator $\Pi$.
\end{surfacecohomology}

\thmref{thm:Psihomology} will follow from Lemmas~\ref{lem:Psikernel} and \ref{lem:Psiimage}, which describe the kernel and image of $\Psi$.

\begin{lemma}\label{lem:Psikernel}
	If $\Psi_k: \Omega^k(\partial M) \to \Omega^{k-1}(\partial M)$ is the restriction of $\Psi$ to the space of $k$-forms on $\partial M$, then $\ker \Psi_k$ is a direct sum of three spaces:
	\begin{enumerate}
		\item \label{enum:Psikernel1} The pullbacks of harmonic Neumann fields
		\[
			i^*\hk_N(M) = \ker \Phi_k.
		\]
		
		\item \label{enum:Psikernel2} The space
		\[
			\ker G_k \cap i^*\left((\mathcal{C}^k(M))^\perp \right),
		\]
		which consists of the pullbacks of $k$-forms with conjugates on $M$ which are perpendicular to the space of closed forms.
		
		\item \label{enum:Psikernel3} A space isomorphic to $H^{k+1}(M, \partial M; \mathbb{R})$.
	\end{enumerate}
\end{lemma}

The operator $G_k$ is the restriction to $\Omega^k(\partial M)$ of the operator $G$ defined in \eqref{eqn:Gdef}.

\begin{lemma}\label{lem:Psiimage}
	The image of the operator $\Psi_{k+1}: \Omega^{k+1}(\partial M) \to \Omega^k(\partial M)$ is precisely the space
	\[
		\ker G_k \cap \,i^*\!\left((\mathcal{C}^k(M))^\perp \right).
	\]
\end{lemma}

\begin{proof}[Proof of \lemref{lem:Psikernel}]
	Suppose $\varphi \in \Omega^k(\partial M)$ such that $\Psi \varphi = 0$. Then, if $\omega \in \Omega^k(M)$ solves the boundary value problem \eqref{eqn:PhiPsiBVP}, we have that
	\begin{equation}\label{eqn:pfPsikernel1}
		0 = \Psi \varphi = i^*\delta \omega.
	\end{equation}
	Using the Hodge-Morrey decomposition of $\Omega^k(M)$ \cite[Theorem 2.4.2]{Schwarz},
	\begin{equation}\label{eqn:pfPsikernel2}
		\omega = \delta \xi + \kappa + d\zeta \in c\mathcal{E}^k_N(M) \oplus \hk(M) \oplus \mathcal{E}^k_D(M),
	\end{equation}
	where
	\begin{align*}
		c\mathcal{E}_N^k(M) & = \{\omega \in \Omega^k(M) : \omega = \delta \xi \text{ for some } \xi \in \Omega^{k+1}(M) \text{ with } i^*\!\star \xi = 0 \} \\
		\hk(M) & = \{\omega \in \Omega^k(M) : d\omega = 0, \delta \omega = 0 \} \\
		\mathcal{E}_D^k(M) & = \{\omega \in \Omega^k(M) : \omega = d\zeta \text{ for some } \zeta \in \Omega^{k-1}(M) \text{ with } i^*\zeta = 0 \}.
	\end{align*}
	Equations \eqref{eqn:pfPsikernel1} and \eqref{eqn:pfPsikernel2} imply that
	\begin{equation}\label{eqn:deltadzeta}
		0 = i^*\delta \omega = i^*\delta(\delta \xi + \kappa + d\zeta) = i^*\delta d \zeta.
	\end{equation}
	
	Since $\delta d \zeta$ is co-exact and since the space of co-exact $k$-forms is precisely the orthogonal complement of the space of $k$-forms satisfying a Dirichlet boundary condition, \eqref{eqn:deltadzeta} implies that $\delta d \zeta = 0$. Hence, $d\zeta$ is co-closed---but $\mathcal{E}^k_D(M)$ is precisely the orthogonal complement of the space of co-closed $k$-forms, so it follows that $d\zeta = 0$.
	
	Therefore,
	\[
		\omega = \delta \xi + \kappa
	\]
	is co-closed. Since both $\omega$ and $\delta \xi \in c\mathcal{E}^k_N(M)$ satisfy a Neumann boundary condition, $\kappa$ must be a harmonic Neumann field. Moreover, since both $\omega$ and $\kappa$ are harmonic, it follows that $\delta \xi$ is harmonic. Hence,
	\[
		\omega = \delta \xi + \kappa \in (c\mathcal{E}^k_N(M) \cap \ker \Delta) \oplus \hk_N(M)
	\]
	and so
	\begin{equation}\label{eqn:istaromega1}
		\varphi = i^*\omega \in i^*(c\mathcal{E}^k_N(M) \cap \ker \Delta) + i^*\hk_N(M).
	\end{equation}
	Conversely, forms in this space are clearly in the kernel of $\Psi$.
	
	In \eqref{eqn:istaromega1} the sum of spaces is not, \emph{a priori}, direct, but directness of the sum follows immediately from the fact that harmonic forms are uniquely determined by their boundary values \cite[Theorem 3.4.10]{Schwarz}.
	
	The term $i^*\hk_N(M) = \ker \Phi_k$ in \eqref{eqn:istaromega1} is exactly the space described in \eqref{enum:Psikernel1}, so the lemma will follow from showing that $i^*(c\mathcal{E}^k_N(M) \cap \ker \Delta)$ is the direct sum of the spaces described in \eqref{enum:Psikernel2} and \eqref{enum:Psikernel3}.
	
	Suppose, then, that $\varphi \in i^*(c\mathcal{E}^k_N(M) \cap \ker \Delta)$; i.e., that $\omega = \delta \xi$. Since $0 = \Delta \omega = \Delta \delta \xi$, we know that
	\[
		0 = (d\delta + \delta d)\delta \xi = \delta d \delta \xi,
	\]
	so $d\delta \xi$ is co-closed, meaning that $d\delta \xi \in \mathcal{H}^{k+1}(M)$; specifically, $d\delta \xi \in \mathcal{EH}^{k+1}(M)$. On the other hand, for any $d\gamma \in \mathcal{EH}^{k+1}(M)$, there is a unique choice of primitive $\gamma$ that is in $c\mathcal{E}^k_N(M) \cap \ker \Delta$. Hence,
	\[
		c\mathcal{E}^k_N(M) \cap \ker \Delta \simeq \mathcal{EH}^{k+1}(M).
	\]
	In turn, since forms in $c\mathcal{E}^k_N(M) \cap \ker \Delta$ are uniquely determined by their pullbacks to the boundary, this implies that
	\[
		i^*(c\mathcal{E}^k_N(M) \cap \ker \Delta) \simeq \mathcal{EH}^{k+1}(M).
	\]
	
	Applying the Hodge star to the space $c\mathcal{E}^k_N(M) \cap \ker \Delta$ yields Cappell, DeTurck, Gluck, and Miller's space $\mathrm{EHarm}^{n-k}$. Thinking in those terms, $\delta \xi \in c\mathcal{E}^{k}_N(M)$ is a harmonic, co-exact form, but the primitive $\xi$ is not necessarily harmonic. There are two possibilities:
	\begin{description}
		\item[Case 1] If $\xi$ is harmonic, then
		\[
			0 = \Delta \xi = (d\delta + \delta d)\xi = d\delta \xi + \delta d \xi,
		\]
		meaning that $d\delta \xi = -\delta d \xi$ is both exact and co-exact. Since $\Delta \delta \xi = 0$, this means that $\delta \xi$ has a conjugate form (in the sense of \cite[Section 5]{BS}). This implies that $i^*\delta \xi \in \ker G_k$ \cite[Theorem 5.1]{BS}. Since $\delta \xi$ is orthogonal to the space of closed $k$-forms on $M$, we have
		\[
			\varphi = i^*\delta \xi \in \ker G_k \cap i^*\!\left((\mathcal{C}^k(M))^\perp \right),
		\]
		which is the space in \eqref{enum:Psikernel2}.
		
		Conversely, if $\varphi \in \ker G_k \cap \,i^*\!\left((\mathcal{C}^k(M))^\perp \right)$, then $\varphi = i^*\delta \xi$ for some $\delta \xi \in c\mathcal{E}^k_N(M)$ which has a conjugate form. This implies that $d\delta \xi$ is both exact and co-exact, and it is straightforward to check that $\xi$ can be chosen to be harmonic.
		
		\item[Case 2] If $\xi$ is not harmonic, then it belongs to the space
		\[
			\mathcal{N}^k := \{\delta \xi \in c\mathcal{E}^k_N(M) \cap \ker \Delta : \Delta \xi \neq 0\}.
		\]
		This space is isomorphic to $H^{k+1}(M, \partial M; \mathbb{R})$ \cite[Lemma 3]{CDGM}, and so $i^*\mathcal{N}^k$ is the space given in \eqref{enum:Psikernel3}.
	\end{description}
	
	The directness of the sum
	\[
		\left(\ker G_k \cap \,i^*\!\left((\mathcal{C}^k(M))^\perp \right) \right) + i^*\mathcal{N}^k
	\]
	again follows from the fact that harmonic forms are uniquely determined by their boundary values.
\end{proof}

We can now determine the image of $\Psi_{k+1}$.

\begin{proof}[Proof of \lemref{lem:Psiimage}]
	Suppose $\vartheta \in \Omega^k(\partial M)$ such that $\vartheta = \Psi \varphi$ for some $\varphi \in \Omega^{k+1}(\partial M)$. If $\omega \in \Omega^{k+1}(M)$ solves the boundary value problem \eqref{eqn:PhiPsiBVP}, then $\vartheta = \Psi \varphi = i^*\delta \omega$.
	
	Since $\omega$ satisfies a Neumann boundary condition,
	\[
		\delta \omega \in c\mathcal{E}^k_N(M).
	\]
	Moreover, since $\Delta$ commutes with the co-differential,
	\[
		\Delta \delta \omega = \delta \Delta \omega = 0,
	\]
	and so
	\[
		\delta \omega \in c\mathcal{E}^k_N(M) \cap \ker \Delta.
	\]
	Since $\omega$ is itself harmonic, this is precisely the situation described in Case 1 of the proof of \lemref{lem:Psikernel}, so
	\[
		\vartheta = i^*\delta \omega \in \ker G_k \cap \,i^*\!\left((\mathcal{C}^k(M))^\perp \right).
	\]
	
	Conversely, if $\vartheta = i^*\delta \zeta$ for $\delta \zeta \in c\mathcal{E}^k_N(M) \cap \ker \Delta$ with $\zeta$ harmonic, then
	\[
		\Delta \zeta = 0 \quad \text{and} \quad i^*\!\star \zeta = 0,
	\]
	so $\vartheta = i^*\delta \zeta = \Psi i^*\zeta$ is in the image of $\Psi$.
\end{proof}

\begin{corollary}\label{cor:domainrangePhiPsi}
	\[
		\ker \Phi_k \subset \ker \Psi_k \quad \text{and} \quad \mathrm{im}\, \Psi_k \subset \mathrm{im}\, \Phi_{n-k}.
	\]
\end{corollary}
\begin{proof}
	The fact that $\ker \Phi_k \subset \ker \Psi_k$ is an immediate consequence of \lemref{lem:Psikernel}.
	
	Now, suppose $\varphi \in \mathrm{im}\, \Psi_k$. Then, by \lemref{lem:Psiimage}, $\varphi \in \ker G_{k-1}$, meaning $\varphi = i^*\omega$ for $\omega \in \Omega^{k-1}(M)$ satisfying
	\[
		\Delta \omega = 0, \quad \delta \omega = 0, \quad \text{and} \quad d\omega = \star d\eta
	\]
	for some $\eta \in \Omega^{n-k-1}(M)$ with $\Delta \eta = 0$ and $\delta \eta = 0$ \cite[Theorem~5.1]{BS}. Therefore, for any $\lambda_N \in \mathcal{H}_N^{n-k}(M)$,
	\begin{equation}\label{eqn:imPsi1}
		\int_{\partial M} \varphi \wedge i^*\lambda_N = \pm \int_{\partial M} i^*\omega \wedge i^*(\star \star \lambda_N) = \pm \left[ \langle d\omega, \star \lambda_N \rangle_{L^2(M)} - \langle \omega, \delta \star \lambda_N \rangle_{L^2(M)} \right]
	\end{equation}
	by Green's formula. The second term on the right hand side vanishes since $\lambda_N$ is closed, while the first is equal to
	\begin{equation}\label{eqn:imPsi2}
		\langle \star d\eta, \star \lambda_N \rangle_{L^2(M)} = \langle d\eta, \lambda_N \rangle_{L^2(M)} = 0.
	\end{equation}
	The first equality above is due to the fact that $\star$ is an isometry and the second follows because $\mathcal{H}^{n-k}_N(M)$ is orthogonal to the space of exact forms on $M$.
	
	Putting \eqref{eqn:imPsi1} and \eqref{eqn:imPsi2} together shows that
	\[
		\int_{\partial M} \varphi \wedge i^*\lambda_N = 0
	\]
	for any $\lambda_N \in \mathcal{H}_N^{n-k}(M)$, so \lemref{lem:kerPhi} implies that $\varphi \in \mathrm{im}\, \Phi_{n-k}$, as desired.
\end{proof}


\section{Cochain maps and the adjoint of $\Psi$} 
\label{sec:cochain_maps_and_adjoints}
Since $\Psi$ is a chain map whose homologies are interesting, it seems natural to try to find associated cochain maps and compute their cohomologies. In fact, there are two such maps,
\[
	\til{\Psi} := (-1)^{k(n-1)}\star_\partial \Psi \star_\partial \quad \text{and} \quad \Theta := (-1)^{(k+1)(n-1)} \Phi \Psi \Phi.
\]
By definition both are maps $\Omega^k(\partial M) \to \Omega^{k+1}(\partial M)$.

\subsection{The operator $\til{\Psi}$} 
\label{sub:the_operator_psi}
The fact that $\til{\Psi}^2 = 0$ is immediate:
\[
	\til{\Psi}^2 = \pm \star_\partial \Psi \star_\partial \star_\partial \Psi \star_\partial = \pm \star_\partial \Psi^2 \star_\partial = 0,
\]
since $\Psi^2 = 0$.

Let $\til{\Psi}^k$ be the restriction of $\til{\Psi}$ to $\Omega^k(\partial M)$. Since $\star_\partial$ is an isomorphism,
\[
	\ker \til{\Psi}^k \simeq \ker \Psi_{n-k-1} \quad \text{and} \quad \mathrm{im}\, \til{\Psi}^{k-1} \simeq \mathrm{im}\, \Psi_{n-k},
\]
and so
\begin{equation}\label{eqn:Psiduality}
	H^k(\Omega^*(\partial M), \til{\Psi}) \simeq H_{n-k-1}(\Omega^*(\partial M), \Psi).
\end{equation}
Thus, we can use \thmref{thm:Psihomology} to determine the cohomology groups of $\til{\Psi}$.

\begin{proposition}\label{prop:tilPsihomology}
	The cohomology groups of the cochain complex $(\Omega^*(\partial M), \til{\Psi})$ are
	\[
		H^k(\Omega^*(\partial M), \til{\Psi}) \simeq H^{n-k}(M; \mathbb{R}) \oplus H^{n-k-1}(M, \partial M; \mathbb{R})
	\]
\end{proposition}

The obvious guess, suggested by experience with $\Lambda$ and by the duality given in \eqref{eqn:Psiduality}, is that $\til{\Psi}$ is the adjoint of $\Psi$.

\begin{proposition}\label{lem:Psiadjoint}
	$\til{\Psi}$ is the adjoint of $\Psi$.
\end{proposition}

\begin{proof}
	The proof follows along similar lines to the proof that $\Lambda^* = \star_\partial \Lambda \star_\partial$ \cite[p. 132]{BS}.
	
	Let $\varphi \in \Omega^k(\partial M)$ and $\psi \in \Omega^{n-k}(\partial M)$. Suppose $\omega \in \Omega^k(M)$ solves the boundary value problem \eqref{eqn:PhiPsiBVP} and that $\eta \in \Omega^{n-k}(M)$ solves the equivalent boundary value problem for $\psi$.
	
	The key step is to show that
	\begin{equation}\label{eqn:keyadjoint}
		(-1)^{k+1}\int_{\partial M} \varphi \wedge \Psi \psi = (-1)^{kn+n+1} \int_{\partial M} \psi \wedge \Psi \varphi. 
	\end{equation}
	Provided this is true, we can re-write the above equation as
	\[
		(-1)^{kn+k+1} \langle \varphi, \star_\partial \Psi \psi \rangle_{L^2(\partial M)} = -\langle \psi, \star_\partial \Psi \varphi \rangle_{L^2(\partial M)}
	\]
	or, equivalently,
	\[
		\langle \varphi, \star_\partial \Psi \psi \rangle_{L^2(\partial M)} = (-1)^{k(n-1)} \langle \psi , \star_\partial \Psi \varphi \rangle_{L^2(M)}.
	\]
	Letting $\psi = \star_\partial \psi'$, this becomes
	\[
		\langle \psi, \star_\partial \Psi \star_\partial \psi' \rangle_{L^2(\partial M)} = (-1)^{k(n-1)} \langle \star_\partial \psi', \star_\partial \Psi \varphi \rangle_{L^2(\partial M)} = (-1)^{k(n-1)} \langle \psi', \Psi \varphi \rangle_{L^2(\partial M)},
	\]
	since $\star_\partial$ is an isometry. Therefore,
	\[
		\Psi^* = (-1)^{k(n-1)} \star_\partial \Psi \star_\partial = \til{\Psi},
	\]
	as desired.
	
	To prove \eqref{eqn:keyadjoint} we note that, by Green's formula,
	\begin{align}
		\nonumber \int_{\partial M} \varphi \wedge \Psi \psi = \int_{\partial M} i^*\omega \wedge i^*\delta \eta & = (-1)^{n(k+1)+n+1} \int_{\partial M} i^*\omega \wedge i^*(\star d \star \eta) \\
		\label{eqn:green} & = (-1)^{kn + 1} \left(\langle d\omega, d\star\eta\rangle_{L^2(M)} - \langle \omega , \delta d \star \eta \rangle_{L^2(M)} \right).
	\end{align}
	Notice that
	\[
		-\langle \omega, \delta d \star \eta \rangle_{L^2(M)} = \langle \omega , d\delta \star \eta \rangle_{L^2(M)}
	\]
	since $0 = \star \Delta \eta = \Delta \star \eta = d\delta \star \eta + \delta d \star \eta$. In turn,
	\[
		\langle \delta \omega , \delta \star \eta \rangle_{L^2(M)} = \langle \omega , d \delta \star \eta\rangle_{L^2(M)} - \int_{\partial M} i^*\delta \star \eta \wedge i^*\!\star \omega.
	\]
	Since $i^*\!\star \omega = 0$, the second term on the right hand side vanishes. Therefore, we can re-write \eqref{eqn:green} as
	\begin{equation}\label{eqn:green2}
		\int_{\partial M} \varphi \wedge \Psi \psi = (-1)^{kn+1} \left(\langle d\omega, d \star \eta \rangle_{L^2(M)} + \langle \delta \omega , \delta \star \eta \rangle_{L^2(M)} \right).
	\end{equation}

Completely analogous reasoning yields the expression
\begin{equation}\label{eqn:green3}
	\int_{\partial M} \psi \wedge \Psi \varphi = (-1)^{kn+n+1} \left(\langle d\eta, d\star\omega \rangle_{L^2(M)} + \langle \delta \eta, \delta \star \omega \rangle_{L^2(M)}\right)
\end{equation}
Therefore, \eqref{eqn:keyadjoint}  follows from \eqref{eqn:green2} and \eqref{eqn:green3} because
\begin{align*}
	\langle d\omega, d\star \eta \rangle_{L^2(M)} & = \langle \star d\omega, \star d \star \eta \langle_{L^2(M)} = (-1)^{k(n+1)} \langle \delta \star \omega, \delta \eta \rangle_{L^2(M)} \\
	\langle \delta \omega, \delta \star \eta \rangle_{L^2(M)} & = \langle \star \delta \omega, \star \delta \star \eta \rangle_{L^2(M)} =  (-1)^{k(n+1)} \langle d \star \omega, d \eta \rangle_{L^2(M)}
\end{align*}
(the first equality in each line is due to the fact that $\star$ is an isometry).
\end{proof}


\subsection{The operator $\Theta$} 
\label{sub:the_operator_theta}
The are several different equivalent ways of expressing the operator $\Theta = (-1)^{(k+1)(n+1)}\Phi \Psi \Phi$. Using \eqref{eqn:PhiPsi},
\begin{equation}\label{eqn:Thetaalt1}
	\Theta = (-1)^{(k+1)(n+1)} \Phi \Psi \Phi = (-1)^{kn}d_\partial \Phi^2.
\end{equation}
On the other hand, using \eqref{eqn:PsiPhi},
\begin{equation}\label{eqn:Thetaalt2}
	\Theta = (-1)^{(k+1)(n+1)} \Phi \Psi \Phi = (-1)^{n(k+1)}\Phi^2 d_\partial.
\end{equation}
Finally, combining \eqref{eqn:Phi2} with \eqref{eqn:Thetaalt2} yields
\begin{equation}\label{eqn:Thetaalt3}
	\Theta = (-1)^{n(k+1)}\Phi^2 d_\partial = (d_\partial \Psi + \Psi d_\partial) d_\partial = d_\partial \Psi d_\partial.
\end{equation}

This last expression makes it clear that $\Theta$ is a cochain map:
\[
	\Theta^2 = d_\partial \Psi d_\partial d_\partial \Psi d_\partial = 0.
\]

\begin{proposition}\label{prop:Thetacohomology}
	The cohomology of the cochain complex $(\Omega^*(\partial M), \Theta)$ is given, up to isomorphism, by
	\[
		H^k(\Omega^*(\partial M), \Theta) \simeq H^{k+1}(M, \partial M; \mathbb{R}) \oplus H^k(M; \mathbb{R}).
	\]
\end{proposition}

Notice that $(\Omega^*(\partial M), \Theta)$ has the same cohomology as $(\Omega^*(\partial M), \til{\Psi})$.

We omit the proof of \propref{prop:Thetacohomology}, which is somewhat long and technical, though not particularly difficult. Two perhaps surprising consequences are:
\begin{enumerate}
	\item Since $\Theta$ has the same cohomology as $\til{\Psi}$, the homology of $\Psi$ can be completely recovered from that of $\Theta$. However, by \eqref{eqn:Thetaalt3}, $\Theta = d_\partial \Psi d_\partial$, so pre- and post-composing $\Psi$ by $d_\partial$ does not change the (co)homology.
	
	\item By \eqref{eqn:Thetaalt1} and \eqref{eqn:Thetaalt2},
	\[
		\Theta = \pm d_\partial \Phi^2 = \pm \Phi^2 d_\partial.
	\]
	Hence, the homology of $\Psi$ is completely determined by the operator $\Phi$, and the results of Corollaries~\ref{cor:echo} and \ref{cor:surfacecohomology} depend only on $\Phi$. In that spirit, the following is a restatement of the $k=0$ case of Corollary~\ref{cor:echo}.
\end{enumerate}

\begin{n-1cohomologyPhi}
	A copy of the cohomology group $H^{n-1}(M; \mathbb{R})$ is distinguished by the operator $\Phi$ inside $\Omega^0(\partial M)$, the space of smooth functions on $\partial M$. Specifically,
	\[
		\ker(d_\partial \Phi^2)/\ker \Phi \simeq H^{n-1}(M; \mathbb{R}).
	\]
\end{n-1cohomologyPhi}



\end{document}